\documentclass[12pt]{amsart}

\usepackage{hyperref}
\usepackage{geometry}
\usepackage{amsfonts}
\usepackage{amsmath}
\usepackage{amssymb}
\usepackage{amsthm}
\usepackage{soul}	
 \usepackage{xcolor}
 
\usepackage{bbold } % fonction indicatrice

\newtheorem{thm}{Theorem}
\newtheorem{theo}{Theorem}[section]
\newtheorem{defi}{Definition}[section]

\newtheorem{prop}{Proposition}[section]
\newtheorem{lem}{Lemma}[section]

\title[Complex manifolds of Sobolev mappings and a Hartogs-type theorem]{Complex manifolds of Sobolev mappings and a Hartogs-type theorem in loop spaces}
\author{M. Anakkar}
\thanks{mohammed.anakkar@univ-lille.fr}

\date{\today}
\newcommand{\cc}{\mathbb{C}}
\newcommand{\rr}{\mathbb{R}}
\newcommand{\calx}{\mathcal{X}}
\newcommand{\calu}{\mathcal{U}}
\newcommand{\calv}{\mathcal{V}}
\newcommand{\calb}{\mathcal{B}}
\newcommand{\calk}{\mathcal{K}}

\def\deff{{:=}}
\def\ie{{\textit{i.e.}}}
\def\d{{\partial}}
\newcommand{\norm}[1]{\left\Vert #1\right\Vert}
\def\eqqno(#1){\label{(#1)}}
\def\eqqref(#1){(\ref{(#1)})}
\def\slsf{\slshape \sffamily }

\def\sli{{\sl i)} } 
\def\slii{{\sl i$\!$i)} }

\def\ext{\sf{ext}}

\newcommand{\llbra}{[\![} % intervalle entier gauche
\newcommand{\rrbra}{]\!]} % intervalle entier droit
\begin{document}

\begin{abstract}
We recall the complex structure on the generalised loop spaces $W^{k,2}(S,X)$, where $S$ is a compact 
real manifold with boundary and $X$ is a complex manifold, and prove a Hartogs-type extension theorem 
for holomorphic maps from certain domains in generalized loop spaces. 
\end{abstract}

\maketitle

\tableofcontents

\section*{Introduction}
All manifolds in this paper are supposed to be Hausdorff and second countable. 
 In the first section we recall the notion of the Sobolev class $W^{k,2}(\Omega)$ for a domain $\Omega \subset \rr^n$ as well as its basic properties. We explain that a composition of a smooth map with a map of class $W^{k,2}$ is still in the class $W^{k,2}$. Namely we prove the following theorem which goes back to \cite{M}.
 \begin{thm}
 Let $k > \frac{n}{2}$, then for every $u \in W^{k,2}_{\text{loc}}(\Omega , \rr^m)$ and every $f \in \mathcal{C}^k(\rr^m)$ the function $f \circ u$ is in $W^{k,2}_{\text{loc}}(\Omega)$.
\end{thm}
This permits us to define correctly the space of Sobolev maps $W^{k,2}(S,X)$ between real manifolds $S$ and $X$ 
and provide the natural structure of a Hilbert manifold on this space. In the second section following Lempert 
\cite{L} we discuss the complex Hilbert structure of the Sobolev manifold $W^{k,2}(S,X)$, where $X$ is now a complex manifold. S everywhere is a compact real manifold with boundary. 
  
\smallskip For positive integers $q \geqslant 1 $, $n \geqslant 1$
and real $r \in ]0,1[$ the $q$-concave Hartogs figure  in $\cc^{q+n}$ is defined as
\begin{equation}
\label{(hart-qn)}
H_q^n(r):= \left(\Delta^q\times \Delta^n(r)\right)\cup \left(A^q_{1-r ,1}\times \Delta^n\right).
\end{equation}
where $A^q_{1-r,1}= \Delta^q \backslash \bar \Delta^q_{1-r}$. Here $\Delta^q_r$ stands for the polydisk in $\cc^q$ centered at zero of radius $r$. The envelope of holomorphy of $H_q^n(r)$ is $\Delta^{q+n}$.
We say that a complex manifold $X$ is $q$-Hartogs if every 
holomorphic mapping $f:H_q^1(r) \to X$ extends to a holomorphic mapping $\tilde f:\Delta^{q+1}\to X$.
If the same holds for a complex Hilbert manifold $\calx$ we say that $\calx$ is $q$-Hilbert-Hartogs.
We proved in \cite{A-Z} that if $\calx$ is $q$-Hilbert-Hartogs then every holomorphic mapping $f:H_q^n(r)
\to \calx$ extends to a holomorphic mapping $\tilde f : \Delta^{q+n}\to \calx$. For finite dimensional $X$ 
this was proved in \cite{I}.
In the last section of this paper we prove the following Hartogs-type extension theorem.
\begin{thm} 
\label{thm 2}
Let $\calx$ be a $q$-Hilbert-Hartogs manifold. Then every holomorphic map \\ $F:W^{k,2}(S,H_q^n(r))\to \calx$ 
extends to a holomorphic  map $\tilde F : W^{k,2}(S, \Delta^q\times \Delta^n) \to \calx$.
\end{thm}

This statement gives us an example of pairs of open sets $\calu\subsetneqq \hat \calu$ 
in a complex Hilbert manifold such that holomorphic mappings with values in $q$-Hilbert-Hartogs manifolds 
extend from $\calu $ to $\hat\calu$. It shows that  $\hat\calu\deff W^{k,2}(S, \Delta^q \times \Delta^n)$ is, 
in some sense, the ``envelope of holomorphy'' of $\calu\deff W^{k,2}(S,H_q^n(r))$. 
For compact $S$ without boundary this theorem was proved in \cite{A-I}.

\section{Sobolev maps between smooth manifolds}
Our goal in this section is to define the space of Sobolev $W^{k,2}$-maps from a manifold $S$ to a manifold $X$
and explain that this space possesses a natural topology.
Let $\Omega \subset \rr^n$ be a domain. 
\begin{defi} 
The Sobolev class $W^{k,2}(\Omega)$ 
is defined as follows.
$$ W^{k,2}(\Omega):= \left\lbrace u\in L^2(\Omega) \ | \  \forall \alpha \in \mathbb{N}^n, \ |\alpha| \leqslant k, \  D^{\alpha} u \in L^2(\Omega)  \right\rbrace , $$
where $L^2(\Omega)$ is the Hilbert space of equivalence classes of real valued square integrable functions with respect to the Lebesgue measure in $\rr^n$. 
\end{defi}

The space $W^{k,2}(\Omega)$ endowed with the scalar product 
\[
 <u, v>_{W^{k,2}(\Omega)}= \sum_{|\alpha|\leqslant k}{<D^{\alpha} u \ , \ D^{\alpha} v >_{L^2(\Omega)}}  
\]
is a Hilbert space. Recall that the Fourier transform of a function from the Schwarz space $u \in \mathcal{S}( \rr^n)$ is defined as
$$ \hat{u}(\xi):= \frac{1}{(\sqrt{2 \pi})^n} \int_{\rr^n}{u(x)e^{-i<x,\xi>}d\mu(x)}, $$
where $\mu$ is the Lebesgue measure in $\rr^n$.  
By the density of $\mathcal{S}(\rr^n)$ in $L^2(\rr^n)$ and using the Placherel identity we can extend the Fourier transform to the isometry of $L^2(\rr^n)$ onto itself, i.e. for $f \in L^2(\rr^n)$ one has that $||\hat f||_{L^2(\rr^n)}= \norm{f}_{L^2(\rr^n)}$.  
One can give the characterisation of Sobolev maps by the Fourier transform.
\vspace{4mm}

\hspace{1cm}$ u \in W^{k,2}(\rr^n) \Longleftrightarrow \forall \alpha \in \mathbb{N}^n, |\alpha| \leqslant k, \ \ D^{(\alpha)} u \in L^2(\rr^n) $

\hspace{35,5mm}$\Longleftrightarrow \forall \alpha \in \mathbb{N}^n, |\alpha| \leqslant k, \ \ \xi^\alpha \hat u \in L^2(\rr^n)$

\hspace{35,5mm}$\Longleftrightarrow \ (1+|\xi|^2)^{\frac{k}{2}}\hat u \in L^2(\rr^n)$
 \vspace{4mm}
 
\noindent Moreover the corresponding norms are equivalent. More precisely we have the following.
\begin{prop}
For $u \in W^{k,2}(\rr^n)$ we have the inequalities:
\begin{itemize}
\item[i)] $\norm{u}_{W^{k,2}(\rr^n)} \leqslant \norm{(1+|\xi|^2)^{\frac{k}{2}} \hat{u}}_{L^2(\rr^n)}$;

\item[ii)] there exists $a \in \rr^+$ wich depends on $k$ and $n$ only such that \begin{equation} \label{ineqL2}
\norm{(1+|\xi|^2)^{\frac{k}{2}} \hat{u}}_{L^2(\rr^n)} \leqslant a \norm{u}_{W^{k,2}(\rr^n)}.
\end{equation}
\end{itemize}
\end{prop}

\begin{proof}
First we compute the following:
\begin{eqnarray*}
\norm{u}_{W^{k,2}}^2 &= & {\sum_{|\alpha| \leqslant k} \norm{D^{(\alpha)} u }^2_{L^2}} = {\sum_{|\alpha| \leqslant k} \norm{\widehat{D^{(\alpha)} u }}^2_{L^2}} = {\sum_{|\alpha| \leqslant k} \norm{ \xi_1^{\alpha_1 }\ldots\xi_n^{\alpha_n }  \hat{u} }^2_{L^2}} \\
& = & {\sum_{|\alpha| \leqslant k} \int_{\rr^n}\xi_1^{2\alpha_1 }\ldots \xi_n^{2\alpha_n }  |\hat{u}(\xi)|^2 d\xi }={ \int_{\rr^n} \sum_{|\alpha| \leqslant k}\xi^{2\alpha } |\hat{u}(\xi)|^2 d\xi }.
\end{eqnarray*}

So we can establish the following equality 
$$ \norm{u}_{W^{k,2}} = \sqrt{ \int_{\rr^n} \sum_{|\alpha| \leqslant k}\xi^{2\alpha } |\hat{u}(\xi)|^2 d\xi }.$$

For $(i)$ we have $(1 + |\xi|^2)^k =(1 + \xi_1^2+ \ldots + \xi_n^2)^k= \sum_{|\alpha| \leqslant k} a_\alpha \xi^{2\alpha} \geqslant  \sum_{|\alpha| \leqslant k}\xi^{2\alpha } $ since $a_\alpha \geqslant 1$. And therefore
\begin{eqnarray*}
\norm{u}_{W^{k,2}}&=& \sqrt{ \int_{\rr^n} \sum_{|\alpha| \leqslant k}\xi^{2\alpha } |\hat{u}(\xi)|^2 d\xi } \leqslant \sqrt{ \int_{\rr^n} (1 + |\xi|^2)^k  |\hat{u}(\xi)|^2 d\xi } \leqslant \\
& \leqslant & \norm{(1+|\cdot|^2)^{\frac{k}{2}} \hat{u}}_{L^2}.
\end{eqnarray*}

As for $(ii)$, since $(1 + |\xi|^2)^k = \sum_{|\alpha| \leqslant k} a_\alpha \xi^{2\alpha} \leqslant  a^2 \sum_{|\alpha| \leqslant k}\xi^{2\alpha } $ with $a= \max_{|\alpha|\leqslant k}\sqrt{a_\alpha}$, we obtain 
\begin{eqnarray*}
\norm{(1+|\cdot|^2)^{\frac{k}{2}} \hat{u}}_{L^2} & =&  \sqrt{ \int_{\rr^n} (1 + |\xi|^2)^k  |\hat{u}(\xi)|^2 d\xi } \leqslant a \sqrt{ \int_{\rr^n} \sum_{|\alpha| \leqslant k}\xi^{2\alpha } |\hat{u}(\xi)|^2 d\xi } \\
& \leqslant &a \norm{u}_{W^{k,2}} . 
\end{eqnarray*}
Proposition is proved.
\end{proof}

Denote by $W^{k,2}_{\text{loc}}(\Omega)$ the space  of distributions $u \in \mathcal{D}'(\Omega) $ which are locally in $W^{k,2}$, i.e. for every relatively compact $D \Subset \Omega$  the restriction $u|_D \in W^{k,2}(D)$.
It is clear that for $u \in W^{k,2}_{\text{loc}}(\Omega)$ and a mapping $f:B \to \Omega$ from a domain $B \subset \rr^m$ to $\Omega$ of class $\mathcal{C}^k$ the composition $u \circ f \in W^{k,2}_{\text{loc}}(B)$.  

Recall the Sobolev imbedding theorem: \textit{ if $f \in W^{k,2}(\Omega)$ and $k > m + \frac{n}{2}$, then there exists a $m$-times continuously differentiable function on $\Omega$ that is equal to $f$ almost everywhere}.
 Therefore the condition $k>\frac{n}{2}$ will be always assumed along this text in order to insure that all $u \in W^{k,2}$ are at least continuous.

\begin{theo}
\label{Ck-Sobo} Let $k > \frac{n}{2}$, then for every $u \in W^{k,2}_{\text{loc}}(\Omega)$ and $f \in \mathcal{C}^k(\rr)$ the function $f \circ u$ is in $W^{k,2}_{\text{loc}}(\Omega)$.
\end{theo}
\noindent \textit{Proof.} The proof will be achieved in three steps.
 
\noindent Step 1:
We shall state it in the form of a lemma. 
%prove that the product of two Sobolev maps is still a Sobolev map. 
\begin{lem} \label{lemineq}
For $ \xi, \eta \in \rr^n$ and $t \geqslant 0$ we have the following inequality:
\begin{equation}
(1+|\xi|^2)^t \leqslant 4^t \big( (1+|\xi - \eta |^2)^t + (1+|\eta|^2)^t \big).
\end{equation}
\end{lem}
\begin{proof}
Since $|\xi - \eta + \eta|^2 \leqslant \Big(|\xi - \eta| + |\eta|\Big)^2 \leqslant 2 \Big(|\xi - \eta|^2 + |\eta|^2\Big)$
\begin{eqnarray*}
(1+|\xi|^2)^t & = & (1+|\xi- \eta +\eta |^2)^t \leqslant (1+2|\xi- \eta|^2 +2 |\eta |^2)^t \\
& \leqslant & (4+2|\xi- \eta|^2 +2 |\eta |^2)^t \leqslant 2^t(1+|\xi- \eta|^2 + 1+ |\eta |^2)^t \\
& \leqslant & 2^t \Big( 2 \max(1+|\xi- \eta|^2 , 1+ |\eta |^2)\Big)^t \\
&\leqslant &  4^t \Big(  \max(1+|\xi- \eta|^2 , 1+ |\eta |^2)\Big)^t \\
& \leqslant & 4^t \Big( \max(1+|\xi- \eta|^2 , 1+ |\eta |^2)\Big)^t + 4^t \Big( \min(1+|\xi- \eta|^2 , 1+ |\eta |^2)\Big)^t \\
& =& 4^t ( 1+|\xi- \eta|^2)^t + 4^t(1+ |\eta |^2)^t .
\end{eqnarray*}
Lemma is proved
\end{proof}
\noindent Step 2: Let us state it again as a lemma.
\begin{lem}
\label{product sobo}
If $k > n/2$ then the product of two Sobolev maps $u, v \in W^{k,2}(\rr^n)$ is in $W^{k,2}(\rr^n)$ with the estimate
$$\norm{u v}_{W^{k,2}(\rr^n)} \leqslant C \norm{u}_{W^{k,2}(\rr^n)} \norm{v}_{W^{k,2}(\rr^n)}.$$ 
\end{lem}
\begin{proof}
Since $\widehat{uv} = \hat u * \hat v$ we have 
\vspace{0.5cm}

\noindent 
$\norm{uv}_{W^{k,2}} \leqslant \norm{\widehat{uv}(1+|\cdot|^2)^\frac{k}{2}}_{L^2}=\sqrt{ \int_{\rr^n}|(\hat{u}*\hat{v})(\xi)|^2(1+|\xi|^2)^{k} d\xi } =$
\begin{eqnarray*}
 &= &   \sqrt{ \int_{\xi \in \rr^n}\left|\int_{\eta \in \rr^n}\hat{u}(\xi -\eta) \hat{v}(\eta) (1+|\xi|^2)^{\frac{k}{2}} \ d\eta \right|^2 d\xi }  \\
   %ligne 
 &  \leqslant & \sqrt{ \int_{\xi \in \rr^n}\left(\int_{\eta \in \rr^n}|\hat{u}(\xi -\eta) \hat{v}(\eta)| (1+|\xi|^2)^{\frac{k}{2}} \ d\eta \right)^2 d\xi }  \\
  %ligne 
    & \leqslant & \norm{\int_{ \rr^n} 4^k|\hat{u}(\xi -\eta) \hat{v}(\eta)| (1+|\xi - \eta|^2)^{\frac{k}{2}} d\eta}_{L^2} +\norm{\int_{ \rr^n} 4^k|\hat{u}(\xi -\eta) \hat{v}(\eta)| (1+|\eta|^2)^{\frac{k}{2}} \ d\eta  }_{L^2}  \\
\end{eqnarray*}

The second inequality here follows from lemma \ref{lemineq}.  
In the second integral, we make the change of variable $s= \xi -\eta$ and get 
$$   \norm{\int_{\eta \in \rr^n} 4^k|\hat{u}(\xi -\eta) \hat{v}(\eta)| (1+|\eta|^2)^{\frac{k}{2}} \ d\eta  }_{L^2} = \hspace{8cm}$$
\begin{flalign*} 
 &= \sqrt{\int_{\xi \in \rr^n} \left( 4^k\int_{\eta \in \rr^n} |\hat{u}(\xi -\eta) \hat{v}(\eta)| (1+|\eta|^2)^{\frac{k}{2}} \ d\eta \right)^2 d\xi } \\
&= 4^k \sqrt{\int_{\rr^n} \left( \int_{s \in \rr^n} |\hat{u}(s) \hat{v}(\xi -s)| (1+|\xi -s|^2)^{\frac{k}{2}} \ ds \right)^2 d\xi }  \\
 & \leqslant 
  4^k \norm{\int_{s \in \rr^n} |\hat{u}(s) \hat{v}(\xi -s)| (1+|\xi -s|^2)^{\frac{k}{2}} \ ds}_{L^2} .
\end{flalign*}

Let recall the Minkowski's inequality for integrals, see \cite{S} \S A.1. 
For $p \in ]1, + \infty[$ and $F(x,y)$ a measurable function on the product measure space $X \times Y$ one has
\begin{equation} \label{minkowskineq}
\left( \int_Y \left| \int_X F(x,y)dx \right|^p  dy \right)^{\frac{1}{p}}
\leqslant \int_X \left( \int_Y |F(x,y)|^p dy \right)^{\frac{1}{p}} dx. 
\end{equation}
In other words, we have the inequality : $$\norm{\int_X F(x,y ) dx}_{L^p}
\leqslant \int_X  \norm{F(x,y)}_{L^p} dx. $$
Therefore by Minkowski integral inequality (\ref{minkowskineq}) we obtain
\noindent 
\begin{align*}
& \norm{uv}_{W^{k,2}}  \leqslant &  \\
 \leqslant  & 4^k \norm{\int_{\eta \in \rr^n} |\hat{u}(\xi -\eta) \hat{v}(\eta)| (1+|\xi - \eta|^2)^{\frac{k}{2}} d\eta}_{L^2} +4^k \norm{\int_{s \in \rr^n} |\hat{u}(s) \hat{v}(\xi -s)| (1+|\xi -s|^2)^{\frac{k}{2}} \ ds}_{L^2} & \\ 
    \leqslant &  4^k \int_{\eta \in \rr^n}\norm{ \hat{u}(\xi -\eta) \hat{v}(\eta) (1+|\xi - \eta|^2)^{\frac{k}{2}}}_{L^2} d\eta + 4^k \int_{s \in \rr^n} \norm{\hat{u}(s) \hat{v}(\xi -s) (1+|\xi -s|^2)^{\frac{k}{2}}}_{L^2} ds &	 \\ 
  \leqslant & 4^k \int_{\eta \in \rr^n}| \hat{v}(\eta)|\norm{ \hat{u}(\xi -\eta) (1+|\xi - \eta|^2)^{\frac{k}{2}}}_{L^2} d\eta + 4^k \int_{\rr^n}|\hat{u}(s)| \norm{ \hat{v}(\xi -s) (1+|\xi -s|^2)^{\frac{k}{2}}}_{L^2} ds  & \\ 
       \leqslant  & 4^k \int_{\eta \in \rr^n}| \hat{v}(\eta)|a \norm{u}_{W^{k,2}} d\eta +4^k \int_{s \in \rr^n}|\hat{u}(s)| a\norm{ v}_{W^{k,2}} ds . & \\ 
  \end{align*}

  In the first integral we transform as follows. 
%\begin{small} 
\begin{eqnarray}
 4^k \int_{\eta \in \rr^n}| \hat{v}(\eta)|a \norm{u}_{W^{k,2}} d\eta & \leqslant & 4^k a\norm{u}_{W^{k,2}}  \int_{\eta \in \rr^n}| \hat{v}(\eta)|(1+|\eta|^2)^{\frac{k}{2}}(1+|\eta|^2)^{-\frac{k}{2}}d\eta \leqslant \\ 
 &\leqslant & 4^k a^2\norm{u}_{W^{k,2}} \norm{v}_{W^{k,2}}\norm{(1+|\xi|^2)^{-\frac{k}{2}}}_{L^2}.
\end{eqnarray}
We make the same with the second integral and obtain
  \begin{eqnarray*}
\norm{uv}_{W^{k,2}} &  \leqslant  & 4^k a^2\norm{u}_{W^{k,2}} \norm{v}_{W^{k,2}}\norm{(1+|\xi|^2)^{-\frac{k}{2}}}_{L^2} + 4^k a^2 \norm{u}_{W^{k,2}} \norm{v}_{W^{k,2}}\norm{(1+|\xi|^2)^{-\frac{k}{2}}}_{L^2}  \\ 
&\leqslant &\Big( 2^{2k+1} a^2\norm{(1+|\xi|^2)^{-\frac{k}{2}}}_{L^2} \Big) \norm{u}_{W^{k,2}} \norm{v}_{W^{k,2}} . \\ 
\end{eqnarray*}
Lemma is proved.
\end{proof}

\noindent This step implies that the product of two $W^{k,2}_{\text{loc}}(\Omega)$ maps are still in $W^{k,2}_{\text{loc}}(\Omega)$ and consequently this implies that
\label{P(u)W^k,2}
for any polynomial $P$ and for any $u \in W^{k,2}_{\text{loc}}(\Omega)$ 
the function $P( u )\in W^{k,2}_{\text{loc}}(\Omega) $.

\noindent Step 3 : For $I$ a closed interval one defines the norm on the space $\mathcal{C}^k(I)$ by
$$ \norm{f}_{\mathcal{C}^k(I)} = \sum_{j=0}^k{\sup_{x \in I}|f^{(j)}(x)}| .$$
Remark that:
\begin{itemize}
\item The norm $\norm{f}_{\mathcal{C}^0(I)} = \sup_{x \in I}|f(x)|$  correspond to the usual sup-norm. 
\item The space $\mathcal{C}^k(I)$ endowed with the norm $\norm{\cdot}_{\mathcal{C}^k(I)}$ is a Banach space.  
\end{itemize}

\noindent Recall the approximation theorem of Weirstrass:
	
Let $f \in \mathcal{C}^0(I)$. Then for every $\epsilon>0$, there exists a polynomial $P$ such that for all $x\in [a,b]$, we have $|f(x)-P(x)|<\epsilon$, i.e. $\norm{f-P}_{\mathcal{C}^0(I)}<\epsilon$. 

\noindent From that statement, one can deduce an approximation with respect to the norm $\mathcal{C}^k$:
\begin{prop}
Any function $f:I \to \mathbb{R}$ of class $\mathcal{C}^k$ defined on a closed interval $I$ can be uniformly approximated by polynomials with respect to the norm $\mathcal{C}^k$, i.e.
$$\forall \epsilon>0, \ \ \exists \  P \in \mathbb{R}[X],\ \text{such that} \  \norm{f - P}_{\mathcal{C}^k(I)}< \epsilon. $$
\end{prop}

\begin{proof}Let $I= [a,b]$ with $a,b \in \mathbb{R}$. 
Since $f$ is $\mathcal{C}^k$, use the approximation theorem on $f^{(k)}$. It exists $P_k$ in $\mathbb{R}[X]$ such that $\norm{P_k - f^{(k)}}_{\mathcal{C}^0(I)}< \epsilon$. Then we define 
$$\forall \ l \smallskip \in \llbra 1 , k \rrbra, \  P_{k-l}(x) = \int_a^x P_{k-l+1}(t)dt + f^{(k-l)}(a).$$
We can prove by induction on $ l$ that %the following inequality
\begin{equation}
\label{*}
\forall \ l \smallskip \in \llbra 0 , k \rrbra, \   \forall x \in [a,b], \  \left| P_{k-l}(x) - f^{(k-l)}(x)  \right| \leqslant \epsilon \frac{(x-a)^{l}}{l!}.
\end{equation}
Indeed for $l=0$ the assertion (\ref{*}) is true and we suppose that (\ref{*}) is true for $l-1$. Then we can write the following.
\begin{eqnarray}
| P_{k-l}(x) &-& f^{(k-l)}(x) |=\left| \int_a^x P_{k-l+1}(t)dt + f^{(k-l)}(a) - f^{(k-l)}(x)  \right| = \\
 &= &  \left|  \int_a^x \Big( P_{k-l+1}(t)- f^{(k-l+1)}(t) \Big) dt \right| 
\leqslant   \int_a^x \left| P_{k-l+1}(t)- f^{(k-l+1)}(t) \right|   dt  \\
& \leqslant & \int_a^x \epsilon \frac{(t-a)^{l-1}}{(l-1)!}dt \leqslant  \epsilon \frac{(x-a)^{l}}{l!}.
\end{eqnarray}
For all $ l \in \llbra 0 , k \rrbra$, we obtain that $\norm{ P_0^{k-l}- f^{(k-l)} }_{\mathcal{C}^0(I)} \leqslant \epsilon \frac{(b-a)^{l}}{l!}$ and then
$$
\norm{P_0 - f }_{\mathcal{C}^k(I)} \leqslant \epsilon\sum_{l=0}^k{ \frac{(b-a)^{l}}{l!}}.
$$
\noindent Proposition is proved.
\end{proof}
 Now we can finish	 the proof of the theorem. 
 Fix a relatively compact subdomain $D \Subset \Omega$. Since $u$ is continuous there exists a closed interval $I \supset u(\bar D)$.  
Let $(P_s)_{s \in \mathbb{N}}$ be a sequence of polynomials that converges to $f$ in $\mathcal{C}^k(I)$. Then $(P_s)_{s \in \mathbb{N}}$ is a Cauchy sequence : for every $\epsilon >0$ there exists $N \in \mathbb{N}$ such that for any $l,s >N$ we have $\norm{P_l - P_s}_{\mathcal{C}^k(I)}<\epsilon$. 
By Step 2 for any $s \in \mathbb{N}$ the function $P_s(u)$ is in $W^{k,2}(D)$.  Let us prove that $(P_s(u))_{s \in \mathbb{N}}$ is a Cauchy sequence in the Hilbert space $W^{k,2}(D)$. 
Note $P_{sl} = P_s - P_l$. 

For $\alpha$ with $|\alpha| \leqslant k $ and any polynomial $P$ we have 
\begin{equation} \label{Palpha}
D^\alpha P(u) = \sum_{r=1}^{|\alpha|} \left( \sum_{\underset{\alpha_1,\cdots,\alpha_r \neq 0}{\alpha_1 + ... + \alpha_r = \alpha}} D^{\alpha_1}u \cdots  D^{\alpha_r}u \right) P^{(r)}(u).
\end{equation} 

For $P(X)=X^p$ relation (\ref{Palpha}) becomes 
$$ \sum_{\underset{\alpha_1,\cdots,\alpha_p \neq 0}{\alpha_1 + ... + \alpha_p = \alpha}} D^{\alpha_1}u \cdots  D^{\alpha_p}u 
= \frac{1}{p!} \left( D^{\alpha} (u^p) - 
\sum_{r=1}^{p-1} \left( \sum_{\underset{\alpha_1,\cdots,\alpha_r \neq 0}{\alpha_1 + ... + \alpha_r = \alpha}} D^{\alpha_1}u \cdots  D^{\alpha_{r}}u \right) \frac{p!}{(p-r)!}u^{p-r} \right).$$

Let us prove by induction on $p \leqslant |\alpha|$ that  
$$  \sum_{\underset{\alpha_1,\cdots,\alpha_p \neq 0}{\alpha_1 + ... + \alpha_p= \alpha}} D^{\alpha_1}u \cdots  D^{\alpha_p}u \in L^2(D).$$ 

For $p=1$ it is true because $\frac{\partial^{|\alpha|}u}{\partial x_j^\alpha} \in L^2(D)$ for every $j=1, \dots, n$ since $u \in W^{k,2}(D)$. 
Let suppose that is true for all $s \leqslant p -1$. Then for $p \leqslant |\alpha|$ using the previous expression we obtain 
$$
 \norm{\sum_{\underset{\alpha_1,\cdots,\alpha_p \neq 0}{\alpha_1 + ... + \alpha_p = \alpha}} D^{\alpha_1}u \cdots  D^{\alpha_p}u 
}_{L^2} \leqslant \hspace{10cm} $$
$$ \leqslant 
 \frac{1}{p!}  \norm{D^{\alpha} (u^p)}_{L^2} +
\sum_{r=1}^{p-1}  \norm{ \frac{1}{(p-r)!}u^{p-r} \sum_{\underset{\alpha_1,\cdots,\alpha_r \neq 0}{\alpha_1 + ... + \alpha_r = \alpha}} D^{\alpha_1}u \cdots  D^{\alpha_{r}}u   }_{L^2}. $$

By Lemma \ref{product sobo} the map $u^p \in W^{k,2}(D)$, therefore $D^{\alpha} (u^p) \in L^2(D)$ for $|\alpha|\leqslant k$. 
For the second term since $u$ is continuous we have that $u$ is bounded on $\bar D$ by $M$. 
Then 
$$  \norm{ \frac{1}{(p-r)!}u^{p-r} \sum_{\underset{\alpha_1,\cdots,\alpha_r \neq 0}{\alpha_1 + ... + \alpha_r = \alpha}} D^{\alpha_1}u \cdots  D^{\alpha_{r}}u   }_{L^2} \leqslant  \frac{M^{p-r}}{(p-r)!}
\norm{ \sum_{\underset{\alpha_1,\cdots,\alpha_r \neq 0}{\alpha_1 + ... + \alpha_r = \alpha}} D^{\alpha_1}u \cdots  D^{\alpha_{r}}u   }_{L^2},$$
and by induction hypothesis $$ \norm{ \sum_{\underset{\alpha_1,\cdots,\alpha_r \neq 0}{\alpha_1 + ... + \alpha_r = \alpha}} D^{\alpha_1}u \cdots  D^{\alpha_{r}}u   }_{L^2} < \infty .$$

We obtain from (\ref{Palpha}) that
\begin{eqnarray*}
 \norm{ P_{s}(u)- P_{s}(u)}_{W^{k,2}}^2 &  = &  \norm{ P_{sl}(u)}_{W^{k,2}}^2 =\sum_{|\alpha | \leqslant k} \norm{D^\alpha P_{sl}(u)}_{L^2}^2 \leqslant \\ 
 &\leqslant & \sum_{|\alpha | \leqslant k} \sum_{r=1}^{|\alpha|} \norm{ \sum_{\alpha_1 + ... + \alpha_r = \alpha} D^{\alpha_1}u \cdots  D^{\alpha_r}u }_{L^2}^2 |P_{sl}^{(r)}(u)|^2  \\
  &\leqslant & \sum_{|\alpha | \leqslant k} \sum_{r=1}^{|\alpha|} \norm{ \sum_{\alpha_1 + ... + \alpha_r = \alpha} D^{\alpha_1}u \cdots  D^{\alpha_r}u }_{L^2}^2 \epsilon^2.
\end{eqnarray*}
Then we can conclude that $(P_s(u))_{s \in \mathbb{N}}$ is a Cauchy sequence on the Hilbert space $W^{k,2}(D)$. Therefore the sequence converge to $f(u) \in W^{k,2}(D)$. 
Theorem is proved. 
\qed

\smallskip

\smallskip

\begin{defi}Let $\Omega \subset \rr^n$ be a domain. 
The Sobolev space $W^{k,2}(\Omega, \rr^m)$ is defined as
$$ W^{k,2}(\Omega, \rr^m) = \{ u \in L^2(\Omega, \rr^m) \ | \ \forall \alpha \in \mathbb{N}^n, |\alpha|\leqslant k , D^\alpha u \in L^2(\Omega, \rr^m) \}. $$ 
\end{defi}
As in the one dimensional case one can prove an analogous of the Theorem \ref{Ck-Sobo}. Namely

\begin{theo}
\label{Ck-Sobo-dim} 
Let $k > \frac{n}{2}$, then for every $u \in W^{k,2}_{\text{loc}}(\Omega, \rr^m)$ and $f \in \mathcal{C}^k(\rr^m)$ the function $f \circ u$ is in $W^{k,2}_{\text{loc}}(\Omega)$.
\end{theo} 

\begin{proof}
The idea of the proof follows the one dimensional case. Fix a relatively compact subdomain $D \Subset \Omega$. 

For any polynomial $P(T_1, \dots , T_n)$ and $|\alpha| \leqslant k $ one has
\begin{eqnarray}
D^\alpha(P(u)) &=& \sum_{r=1}^{|\alpha|} \sum_{\underset{\alpha_1,\cdots,\alpha_r \neq 0}{\alpha_1 + ... + \alpha_r = \alpha}} (d^rP)_u \big( D^{\alpha_1}u, \dots ,D^{\alpha_r}u \big)  \\
&=&
\sum_{r=1}^{|\alpha|} \sum_{\underset{\alpha_1,\cdots,\alpha_r \neq 0}{\alpha_1 + ... + \alpha_r = \alpha}} \left( \sum_{i_1, \cdots, i_r=1}^m \frac{\partial^r P}{\partial t_{i_1} \cdots \partial t_{i_r}}(u) D^{\alpha_1}u_{i_1} \dots D^{\alpha_r}u_{i_r}  \right) \\
&=&
\sum_{r=1}^{|\alpha|}  \sum_{i_1, \cdots, i_r=1}^m \frac{\partial^r P}{\partial t_{i_1} \cdots \partial t_{i_r}}(u) \left( \sum_{\underset{\alpha_1,\cdots,\alpha_r \neq 0}{\alpha_1 + ... + \alpha_r = \alpha}} D^{\alpha_1}u_{i_1} \dots D^{\alpha_r}u_{i_r}  \right) \label{previous equation}
\end{eqnarray}

Fix $\alpha \in \mathbb{N}^n$ with $|\alpha| \leqslant k$. 
We prove by induction on $p$ with $1 \leqslant p \leqslant |\alpha| $ that
$$
\forall I=(j_1, \dots, j_p) \in \llbra 1, m \rrbra^p,   \ 
 \sum_{\underset{\alpha_1,\cdots,\alpha_{{p}} \neq 0}{\alpha_1 + ... + \alpha_{{p}} = \alpha}} D^{\alpha_1}u_{j_1} \dots D^{\alpha_p}u_{j_{p}}    \in L^2(D) .$$

For $p=1$ the property is true because $D^\alpha u_j \in L^2(D)$ for  $j=1, \dots, m$   since $D^\alpha u \in L^2(D, \rr^m)$.  
Let suppose that is true for all $s \leqslant p -1$. For $p\leqslant |\alpha| $ one can define
for every $J=(j_1, \dots, j_p) \in \llbra 1, m \rrbra^p$  $\beta = e_{j_1} + \dots + e_{j_p} $ with   $(e_j)_{i=1, ..., m}$ the standard basis in $\rr^m$. Then consider $P(T) =T^\beta= T_1^{\beta_1}\cdots T_m^{\beta_m}$ in relation (\ref{previous equation}) 
\begin{eqnarray*} 
D^\alpha(u^\beta) &=& \sum_{r=1}^{|\alpha|}  \sum_{i_1, \cdots, i_r=1}^m \frac{\partial^r P}{\partial t_{i_1} \cdots \partial t_{i_r}}(u) \left( \sum_{\underset{\alpha_1,\cdots,\alpha_r \neq 0}{\alpha_1 + ... + \alpha_r = \alpha}} D^{\alpha_1}u_{i_1} \dots D^{\alpha_r}u_{i_r}  \right) \\
&=& \sum_{r=1}^{p}  \sum_{i_1, \cdots, i_r=1}^m \frac{\partial^r T^\beta}{\partial t_{i_1} \cdots \partial t_{i_r}}(u) \left( \sum_{\underset{\alpha_1,\cdots,\alpha_r \neq 0}{\alpha_1 + ... + \alpha_r = \alpha}} D^{\alpha_1}u_{i_1} \dots D^{\alpha_r}u_{i_r}  \right) .
\end{eqnarray*}
Note that $\frac{\partial^p T^\beta}{\partial t_{j_1} \cdots \partial t_{j_p}} = \beta !$ and the number of indices $J$ such that $\beta = e_{j_1} + \dots + e_{j_p} $ is equal again to $\beta!$. The other term are equal to zero. 
Then we obtain
\begin{align*}
(\beta!)^2 & \sum_{\underset{\alpha_1,\cdots,\alpha_p \neq 0}{\alpha_1 + ... + \alpha_p = \alpha}} D^{\alpha_1}u_{j_1} \dots D^{\alpha_r}u_{j_p} =&\\
& = D^\alpha(u^\beta) -
 \sum_{r=1}^{p-1}  \sum_{i_1, \cdots, i_r=1}^m \frac{\partial^r T^\beta}{\partial t_{i_1} \cdots \partial t_{i_r}}(u)  \left( \sum_{\underset{\alpha_1,\cdots,\alpha_r \neq 0}{\alpha_1 + ... + \alpha_r = \alpha}} D^{\alpha_1}u_{i_1} \dots D^{\alpha_r}u_{i_r}  \right),
\end{align*}
where
$\beta! = \beta_1 ! \cdots \beta_m !$.

The term $ D^\alpha (u^\beta) = D^\alpha (u_1^{\beta_1} \dots u_m^{\beta_m})$ is in $L^2(D)$ since $u_1 , \dots u_m$ are in $W^{k,2}(D)$ and Lemma \ref{product sobo}. For the other term one can remark first that   $\frac{\partial^r X^\beta}{\partial x_{i_1} \cdots \partial x_{i_r}}(u)$ is a polynomial evaluate on $u$ and since $u$ is bounded on $D$ this polynomial is bounded by $M \in \rr$. Therefore for every $(i_1, \dots, i_r) \in \llbra 1,m \rrbra^r$ one has
$$
\norm{ \frac{\partial^r X^\beta}{\partial x_{i_1} \cdots \partial x_{i_r}}(u) 
 \left( \sum_{\underset{\alpha_1,\cdots,\alpha_r \neq 0}{\alpha_1 + ... + \alpha_r = \alpha}} D^{\alpha_1}u_{i_1} \dots D^{\alpha_r}u_{i_r}  \right) } \leqslant
 M \norm{ \sum_{\underset{\alpha_1,\cdots,\alpha_r \neq 0}{\alpha_1 + ... + \alpha_r = \alpha}} D^{\alpha_1}u_{i_1} \dots D^{\alpha_r}u_{i_r}}_{L^2(D)} 
$$
and by induction hypothesis this is in $L^2(D)$. Then for every $|\alpha| \leqslant k$ 
$$
\forall p \leqslant |\alpha|, \ 
\forall I=(j_1, \dots, j_p) \in \llbra 1, m \rrbra^p,   \ 
 \sum_{\underset{\alpha_1,\cdots,\alpha_{{p}} \neq 0}{\alpha_1 + ... + \alpha_{{p}} = \alpha}} D^{\alpha_1}u_{j_1} \dots D^{\alpha_p}u_{j_{p}}    \in L^2(D). 
$$
Now, since $u$ is continuous there exists a compact set $K \supset u(\bar D)$. 
Let $(P_s)_{s \in \mathbb{N}}$ be a sequence of polynomials that converges to $f$ in $\mathcal{C}^k(K)$. Then $(P_s)_{s \in \mathbb{N}}$ is a Cauchy sequence : for every $\epsilon >0$ there exists $N \in \mathbb{N}$ such that for any $l,s >N$ we have $\norm{P_l - P_s}_{\mathcal{C}^k(K)}<\epsilon$. 
By Lemma \ref{product sobo} for any $s \in \mathbb{N}$ the function $P_s(u)$ is in $W^{k,2}(D)$.  Let us prove that $(P_s(u))_{s \in \mathbb{N}}$ is a Cauchy sequence in the Hilbert space $W^{k,2}(D)$. 
Note $P_{sl} = P_s - P_l$.  

\begin{eqnarray}
\norm{D^\alpha(P_{sl}(u))}_{L^2(D)}^2
& \leqslant &
\sum_{r=1}^{|\alpha|}  \sum_{i_1, \cdots, i_r=1}^m \left| \frac{\partial^r P_{sl}}{\partial t_{i_1} \cdots \partial t_{i_r}}(u) \right|^2 \norm{ \sum_{\underset{\alpha_1,\cdots,\alpha_r \neq 0}{\alpha_1 + ... + \alpha_r = \alpha}} D^{\alpha_1}u_{i_1} \dots D^{\alpha_r}u_{i_r} }_{L^2(D)}^2\\
& \leqslant &
\sum_{r=1}^{|\alpha|}  \sum_{i_1, \cdots, i_r=1}^m \epsilon^2 \norm{ \sum_{\underset{\alpha_1,\cdots,\alpha_r \neq 0}{\alpha_1 + ... + \alpha_r = \alpha}} D^{\alpha_1}u_{i_1} \dots D^{\alpha_r}u_{i_r} }_{L^2(D)}^2 \\
& \leqslant & \epsilon^2
\sum_{r=1}^{|\alpha|}  \sum_{i_1, \cdots, i_r=1}^m \norm{ \sum_{\underset{\alpha_1,\cdots,\alpha_r \neq 0}{\alpha_1 + ... + \alpha_r = \alpha}} D^{\alpha_1}u_{i_1} \dots D^{\alpha_r}u_{i_r} }_{L^2(D)}^2 
\end{eqnarray}
Therefore $(P_s(u))_s$ is a Cauchy sequence in the Hilbert space $W^{k,2}(D)$ so the sequence converge to $f(u) \in W^{k,2}(D)$. 
Theorem is proved.  
\end{proof}

From that one can define Sobolev spaces of maps between manifolds.
\begin{defi}
Let $S$ and $X$ be real manifolds of class $\mathcal{C}^k$. A map $u: S \to X$ is said to be in $W^{k,2}(S,X)$ 
if for every $s \in S$, every coordinate chart $(V, \psi)$ which contains $s$ and every  coordinate 
chart $(U, \phi)$ which contains $u(s)$ one has that $\phi \circ u \circ \psi^{-1} \in W^{k,2}_{loc}$.
\end{defi}

\noindent Theorem \ref{Ck-Sobo-dim} insures that this definition is correct. The space of $W^{k,2}$-maps from $S$ to $X$ we denote as $W^{k,2}(S,X)$. Notice that $W^{k,2}(S,X)$ inherits the natural topology from $W^{k,2}_{loc}$. We shall call this topology the Sobolev topology.

\section{Complex structure on the space of Sobolev maps}
From now on let $S$ be a compact connected n-dimensional real manifold with boundary. Let $X$ be a finite dimensional complex manifold. Our goal in this section to equip the Sobolev space $W^{k,2}(S,X)$ with a natural structure of a 
complex Hilbert manifold.  

For a set $U \subset S \times X$ and $s\in S$ we write
$$U^s:=\{ x \in X \ | \ (s,x) \in U \} \hspace{1cm} \text{and} \hspace{1cm}
\begin{array}{ccccc}
\epsilon^s & : & U^s & \to & U \\
 & & x & \mapsto & (s,x). \\
\end{array}  $$

\begin{lem}
\label{chart-hilbert-manifold}
Given $g \in W^{k,2}(S,X)$ with $k > \frac{n}{2} $, there exists a $W^{k,2}$-diffeomorphism $G$ between a neighborhood $U \subset S \times X$ of the graph $\{ (s,g(s)) \ | s \in S \}$ of $g$ and a neighborhood of the zero section in $g^*TX$ such that
\begin{itemize}

\item[i)] $\{ G(s,g(s)) \ | \ s \in S \}$ is the zero section of $g^*TX$;

\item[ii)] $G^s=G \circ \epsilon^s$ maps $U^s$ biholomorphically on a neighborhood of $0 \in T_{g(s)}X$;

\item[iii)] $dG^s_{g(s)}$ is the identity map.

\end{itemize}
\end{lem}
\begin{proof}
i)
We recall the argument from \cite{L} pointing out the smoothness of $G$. Let $(\Omega_j, \phi_j)$ be an atlas of the complex manifold $X$. Then  the sets $S_j = g^{-1}(\Omega_j) \subset S$ form an open covering of $S$. Consider $U_j \subset S_j \times \Omega_j $ a neighborhood of the graph of $g|_{S_j}$. We can construct locally the diffeomorphism $G_j$ by 
$$
\begin{array}{cccc}
G_j : & U_j & \longrightarrow & g^*TX \\
& (s,x) & \mapsto & \Big( s, (d\phi_j^{-1})_{\phi_j(g(s))} \big[\phi_j(x) - \phi_j(g(s))\big]  \Big) .
\end{array} $$
Notice that $G_j$ is of class $W^{k,2}$ because such is $g$. 
It remains to glue all the $G_j$. Take $\{ \eta_j \}$ a $\mathcal{C}^k$-partition of unity subordinated to the covering $\{ S_j \}$ and define
 $G(s,x)= \sum_{j}\eta_j(s) G_j(s, x)$. Choose then the restriction of $G$ to a suitable $U \subset \bigcup_{j}U_j$. Remark that $G(s,g(s)) =(s,0)$. 

ii) We have that $G_j^s(x)= G_j \circ \epsilon^s(x)= (d\phi_j^{-1})_{\phi_j(g(s))} \big[\phi_j(x) - \phi_j(g(s))\big] $.  The map $G_j^s: U_g^s \to g^*TX$ is holomorphic for every $s \in S_j$ since $x \mapsto (d\phi_j^{-1})_{\phi_j(g(s))} \big[\phi_j(x) - \phi_j(g(s))\big]$ is holomorphic. Here $U_g^s= \{x \in X \ | \ (s,x) \in U_g \}$. Holomorpicity of $G^s(x)= \sum_j \eta_j(s) G^s_j(x)$ follows. 

iii) We compute the differential of $G^s$ and obtain
$$ dG^s_{g(s)}=  (d\phi_j^{-1})_{\phi_j(g(s))} (d\phi_j)_{g(s)} = \text{Id}$$
\end{proof}

For $g \in W^{k,2}(S,X)$ choose $U_g$ and $G$ as in the previous lemma. Those $h \in W^{k,2}(S,X)$ whose graph $\Gamma_h:=\{(s,h(s)) \ | \ s\in S \}$ 
is contained in $U_g$ form a neighborhood $\mathcal{U}_g \subset W^{k,2}(S,X)$ of $g$:
$$\mathcal{U}_g=\left\lbrace h \in W^{k,2}(S,X) \ | \ \Gamma_h \subset U_g \right\rbrace . $$

For $h \in \calu_g$ define the section $\psi_g(h)=G(\cdot,h(\cdot)) \in W^{k,2}(S,g^*TX)$. Thus $\psi_{g}$ is a homeomorphism between $\calu_g$ and an open neighborhood of zero section in $W^{k,2}(S,g^*TX)$. We can define the chart $(\calu_g, \psi_g)$ where local coordinates are in a complex Hilbert space $W^{k,2}(S,g^*TX)$.
Now we need to verify that transition maps are holomorphic.

Let $h\in W^{k,2}(S,X)$ be such that $\Gamma_h \subset U_g \cap U_{g'}$, i.e. $h \in \calu_g \cap \calu_{g'}$. For $s\in S$ we have 
$$ \psi_{g'}(h)(s)=\left(\psi_{g'} \circ \psi_g^{-1} \right)[\psi_g(h)(s)]= \left[{G'}^s \circ {G^s}^{-1} \right][\psi_g(h)(s)] = \left[ G'(s,\cdot) \circ G(s, \cdot)^{-1} \right][\psi_g(h)(s)]$$
Due to item $ii)$ of the Lemma just proved we have that ${G'}^s \circ {G^s}^{-1} $ is a biholomorphism between an appropriate open subsets of $T_{g(s)}X$ and $T_{g'(s)}X$. Therefore the value $\psi_{g'}(h)(s)$ depends holomorphically on $\psi_g(h)(s)$. 

In more details we have a $W^{k,2}$-regular maps $P=G' \circ G^{-1}$ between open sets $V \subset g^*TX$ and $V' \subset {g'}^*TX$ such that for every $s \in S$ $P(s,v)$ holomorphically depends on $v \in V \cap T_{g(s)}X$ where $P(s,v)=(s,P^s(v))$. Let $\calv$ be the open set of sections of $g^*TX$ which are contained in $V$. The same for $\calv'$. We obtain a mapping $\mathcal{P}: \calv \to \calv'$ defined as $\mathcal{P}(h)(s)= P(s,P^s(h(s)))$, where $h$ is a $W^{k,2}$-section of $g^*TX$ contained in $V$. Then $\mathcal{P}$ is holomorphic. This easily follows from Gâteaux differentiability and continuity of $\mathcal{P}$. 
Thus $W^{k,2}(S,X)$ has the structure of a complex Hilbert manifold. 

\begin{lem}
\label{lemp}
Let $D$ and $X$ be finite dimensional complex manifolds and let $S$ be
an $n$-dimensional compact real manifold with boundary. A mapping $F:D\times S\to X$
represents a holomorphic map $F_*:D \to W^{k,2} (S,X)$ if and only if the following holds:
\begin{itemize}

\item[i)] for every $s\in S$ the map $F(\cdot,s):D\to X$ is holomorphic;

\item[ii)] for every $z\in D$ one has $F(z,\cdot)\in W^{k,2} (S,X)$ and the correspondence \\
$D\ni z\mapsto F(z,\cdot)\in W^{k,2} (S,X)$ is continuous
with respect to the Sobolev topology on $W^{k,2}(S,X)$ and the standard
topology on $D$.

\end{itemize}

\end{lem}
\begin{proof} $\Rightarrow$
Given $F_*:D \to W^{k,2}(S,X)$ we construct $F: D \times S \to X$ as follows. For $s \in S$ and $z \in D$ we define $F(z,s)=F_*(z)(s)$. 
By the assumtion about $F_*$ for any $z \in D$ the map $F(z,\cdot)$ is in $W^{k,2}(S,X)$ and the map $z \mapsto F(z,\cdot)=F_*(z)$ is continuous (in fact it is holomorphic).
 To prove the holomorphicity of $F(z,s)$ for a fixed take any $s\in S$ in a neighborhood of some $z_0$ take any chart $(\calu, \psi_g)$ which conains the graph of $F_*(z_0)$. We have 
$\psi_g(F(z,s))(s) = G\big(s,F_*(z)(s)\big)=G(s,F(z,s))$ for $z$ close to $z_0$ by definition of $\psi_g$. This
is holomorphic by the definition of the complex structure of $W^{k,2}(S,X)$.
 Therefore $F: D \times S \to X$ is holomorphic in $z$. 

Conversely given $F: D \times S \to X$ satisfying $i)$ and $ii)$ we can construct $F_*$ as follows. For $z \in D$ we define $F_*(z)=F(z,\cdot)$. Mapping $F_*$ is well defined since $F(z,\cdot) \in W^{k,2}(S,X)$ and is continuous by $ii)$. For the holomorphicity again we take a chart $(\calu, \psi_g)$ and we consider the map $\psi_g \circ F_*$ defined for $z\in D$ by $\psi_g \circ F_*(z) = G\big(\cdot, F(z, \cdot)(\cdot)\big)$. By $i)$ for every $s \in S$ the map $s \mapsto G\big(s, F(z, s)(s)\big)$ is holomorphic by composition of holomorphic maps $F(\cdot, s)$ and $G^s$. Then $ \psi_g \circ F_*$ is holomorphic. 
\end{proof}

\noindent \textbf{Remark.}
If $X$ is a smooth real manifold one can repeat the same contruction to ensure that $W^{k,2}(S,X)$ has a structure of a smooth Hilbert manifold. In fact the holomorphy of item $ii)$ in lemma \ref{chart-hilbert-manifold} should be replaced by smoothness since coordinate charts $\psi_g$ are  smooth.

\section{A Hartogs-type theorem}

In \cite{I} it was proved the following result. 
\begin{theo}
\label{qn-pm} If a complex manifold $X$ is $q$-Hartogs then for any $(p,n)$ with $p\ge q, n\ge 1$  the map $f: H_p^n(r) \to X$ extends holomorphically to $\Delta^{q+n}$. 
\end{theo}
For Hilbert $\calx$ it was proved in \cite{A-Z}. 
The proof of this result for Hilbert $\calx$ lies on the following two statements proved in \cite{A-Z}, and we shall need them here too. Recall first the definition of a $1$-complete neighborhood. 

\begin{defi}\label{1compdef}
A $1$-complete neighborhood of a compact $\calk \subset \calx$ is an open set $\calu \supset \calk$ such that
\begin{itemize}
\item[i)] $\calu$ is contained in a finite union of open coordinate balls centered at points of $\calk$, i.e.
$$\calu \subset \bigcup_{\alpha=1}^{n}{}\calb_\alpha  \ \text{with} \ \calb_\alpha = \phi_\alpha^{-1}(B^\infty) \ \text{and} \ \phi_\alpha^{-1}(0)= k_\alpha  \in \calk,$$

\item[ii)] $\calu$ possesses a strictly plurisubharmonic exhaustion function $\psi: \calu \to [0,t_0)$, i.e.
\begin{itemize}
 \item for every $t<t_0$ one has that
  $\overline{\psi^{-1}\left([0,t)\right)}\subset \calu$.
\end{itemize}
\end{itemize}
\end{defi}

Here by a strictly plurisubhamonic function we mean the following.
\begin{defi}
Let $U$ be an open subset of $\calx$. A function $f \in \mathcal{C}^2(U,\rr)$ is said to be strictly plurisubharmonic on $U$ if the Levi form $\mathcal{L}_{f,a}$ is positive definite for every $a \in U$, i.e
\[ \mathcal{L}_{f,a}(v) > 0 \hspace{3mm} \text{for each} \hspace{3mm} a\in U \hspace{3mm} \text{and}\hspace{3mm} v \in T_a\calx\backslash \{0\}. \]
\end{defi}

In \cite{A-Z} we give a strongest definition of a strictly plurisubharmonic function, see definition 2.1 there. In fact we states that the Levi form staisfies 
\[
{\mathcal{L}_{f,a}(v)} 
\geqslant c(a)||v||^2 \quad \text{for each} \hspace{3mm} a\in U \hspace{3mm} \text{and}\hspace{3mm} v \in T_a\calx\backslash \{0\},
\]
 with $c$ a positive function in $\mathcal{C}^0(U,\rr)$. These definition are not equivalent as it shows the counter-example given by Lempert:
 
\noindent Take for example the function $f$ defined by $f(z) = \sum_{j=1}^\infty{\frac{|z_j|^2}{j}}$ for $z \in l^2$. The Levi form on a point $a \in l^2$ is
$$\forall v  \in l^2, \hspace{3mm} \mathcal{L}_{f,a}(v) = \sum_{j=1}^\infty\frac{1}{j}|v_j|^2 ,$$   
and this cannot be bound from below by $c||v||^2$. Nevertheless plurisubharmonic  functions we need in this paper satisfies this strongest definition.   
 
\begin{theo}
\label{royden-w}
Let $\phi : \bar D \to \calx $ be an imbedded analytic $q$-disk in a complex Hilbert manifold $\calx$. 
Then $\phi (\bar D) $ has a fundamental system of $1$-complete neighborhoods.
\end{theo}

\noindent Here an analytic $q$-disk in a complex Hilbert manifold $\calx$ is a holomorphic imbedding $\phi $ of a neighborhood of a closure of a relatively compact strongly pseudoconvex domain $D\Subset \cc^q$ into $\calx$. 
We shall also need the following lemma from \cite{A-Z}, see Lemmma 3.1 there.
\begin{lem}
\label{cont}
Let $\phi_n:\bar D\to \calx$ be a sequence of analytic $q$-disks in a complex Hilbert manifold $\calx$ and
let $\Phi_n$ be their graphs. Suppose that there exists an analytic disk 
$\phi_0:\bar D\to \calx$ with the graph $\Phi_0$ such that for any neighborhood $\calv\supset \Phi_0$ 
one has $\Phi_n\subset \calv$ for $n>>1$. Then $\phi_n$ converges uniformly on $\bar D$ to $\phi_0$.
\end{lem}
Now we return to the proof of the theorem \ref{thm 2} from the introduction. 
\begin{theo}
\label{env-hol}
Let $\calx$ be a $q$-Hilbert-Hartogs manifold. Then every holomorphic map \\ $F:W^{k,2}(S,H_q^n(r))\to \calx$ 
extends to a holomorphic map $\tilde F : W^{k,2}(S, \Delta^q\times \Delta^n) \to \calx$.
\end{theo}
\proof The proof will be achieved in a number of steps. First we shall construct some ``natural extension'' 
of $F$. After that we shall prove that the extension is continuous and finally that it is holomorphic.

\smallskip\noindent{\slsf Step 1. Natural extension.} Let $f=(f^q,f^n):S\to \Delta^q\times \Delta^n$ 
be an element of $W^{k,2}(S, \Delta^q\times \Delta^n)$. 
Here $f^q$ and $f^n$ are componnents of $f$. 
We want to extend $F$ to $f$. This will be done
along an appropriate analytic disc which passes through $f$. Consider the mapping 
$\phi_f : \bar\Delta^q\times S\to \Delta^q\times \Delta^n$ defined as
\begin{equation}
\label{an-disk1}
\phi_f : (z,s) \to  \left(h_{f^q(s)}(z),f^n(s)\right).
\end{equation}
Here $h_a$ is the following automorphism of $\Delta^q$ interchanging $a=(a_1,...,a_q)$ and $0$:
\[
 h_a(z) = \left(\frac{a_1-z_1}{1-\bar a_1z_1},..., \frac{a_q-z_q}{1-\bar a_qz_q}\right).
\]
Notice that due to the compactness of $f^q(S)\subset \Delta^q$ authomorphisms
$h_{f^q(s)}$ are defined for $z$ in a fixed (independant on $s$) neighborhood of $\bar\Delta^q$,
and therefore such is $\phi_f$. Denote by $\phi_{f*}:\bar\Delta^q\to W^{k,2}(S,\Delta^{q+n})$ the 
analytic $q$-disk in $W^{k,2}(S, \Delta^q\times \Delta^n)$ represented by $\phi_f$, \textit{i.e} $\phi_{f*}(z)
\in W^{k,2}(S, \Delta^{q+1})$ acts as follows
\begin{equation}
\label{an-disk2}
\phi_{f*}(z) : s \to  \left(\frac{f_1^q(s)-z_1}{1-\bar f_1^q(s)z_1},..., \frac{f_q^q(s)-z_q}
{1-\bar f_q^q(s)z_q},f^n(s)\right).
\end{equation}
Here $f^q(s) = (f_1^q(s),...,f_q^q(s))$ is the $q$-component of $f$. Denote by $\Phi_f = 
\phi_{f*}(\bar \Delta^q)$ the image of $\phi_{f*}$. As in lemma \ref{cont} we mark here and everywhere with $*$ analytic disks 
$D\to W^{k,2}(S,X)$ represented by maps $D\times S\to X$ and by capital letters we mark the images 
of these discs. Our $\phi_f$ possesses the following properties:
\begin{itemize}
\item[i)] $\phi_{f*}(0)=f$, \textit{i.e.} is our loop $f$.

\item[ii)] For $z\in\partial \Delta^q$ one has that $\phi_{f*} (z)(S)\subset A^q_{1-r, 1}
\times \Delta^n$, therefore $$\partial \Phi_f \deff \phi_{f*}(\partial\Delta^q)\subset W^{k,2}(S,H_q^n(r)).$$
\end{itemize}
Indeed, for $z$ close to $\d\Delta^q$  some $z_j$ is close to $\d\Delta$ and then the $j$-component of $h_{f^q(s)}(z)$ 
also is close to $\d\Delta$ for all $s\in S$ as required. 
%\begin{rem} \rm
%Strictly speaking for $z\in \d\Delta^q$ $\phi_{f*} (z)(S)\subset   \d\Delta^q\times \Delta^n\not\subset 
%A^q_{1-r, 1}\times\Delta^n$, so a certain rescailing should be made. But we shall not pay attention on 
%such obvious details.
%\end{rem}

Remark that due to the second item above
our map $F$ is defined and holomorphic near the boundary of the analytic $q$-disc $\Phi_f$. In order
to assign to $f$ some ``natural'' value $\tilde F(f)$, which we shall call a ``natural extension''
of $F$, we shall extend $F$ holomorphically to the $q$-disc $\Phi_f\ni f$.

\smallskip 
Consider the analytic $(q+1)$-disk $\psi_*$ in $W^{k,2}(S,\Delta^q\times \Delta^n)$ 
represented by
\begin{equation}
\label{an-disk3}
\psi^t(z,s) \deff  \psi (z,t,s) \deff \left(h_{f^q(s)}(z),
tf^n(s)\right), \quad z\in \bar\Delta^q, \quad |t| \le 1+\delta , \quad s\in S
\end{equation}
for an appropriate $\delta >0$ small enough. Here we use the compactness of $f^n(S)$.
Remark that
\begin{itemize}
\item[iii)] $\psi^0(\cdot , \cdot)$ takes its values in $\Delta^q\times \{0\} \subset H_q^n(r)$ and therefore for $|t|$ small $\psi^t(.,.)$ takes its values in $H_q^n(r)$ by continuity.  

\item[iv)] $\psi^1 = \phi_f$.

\item[v)] For all $t\in \bar\Delta_{1+\delta}$ one has that $\d\Psi^t =\Psi_*^t(\partial\Delta^q) \subset W^{k,2}(S,
H_q^n(r))$ by the second item as before. 
\end{itemize}
Here $\Psi$ and $\Psi^t$ stand for the images of $\psi_*$ and $\psi^t_*$ respectively, i.e. for every t we see $\Psi_*^t$ as an analytic $q$-disk, while $\Psi_*$ as an analytic $(q+1)$-disk. 
Therefore for $\delta >0$ small enough the Hartogs figure
\begin{equation}
\label{(hart-d)}
H_{q}^1(\delta)\deff \{(z,t): ||z||<1 + \delta , |t|<\delta \text{ or }
1-\delta<||z||<1+\delta, |t|<1 + \delta\}
\end{equation}
is mapped by $\psi_*$ to $W^{k,2}(S,H_q^n(r))$ and consequently the composition $F\circ \psi_*
:H_{q}^1(\delta)\to \calx$ is well defined and holomorphic. Due to the assumed $q$-Hartogsness of 
$\calx$ this composition $F\circ \psi_* $ holomorphically extends to the associated polydisc 
$\Delta^{q+1}_{1+\delta}$.  In particular it extends to $\Delta^q\times \{1\}$, \ie  \ $F$ holomorphically
extends onto the $q$-disc $\Phi_f\ni f$.
\begin{equation}
\eqqno(pnt-ext1)
\text{\it Denote by }\quad \widetilde{F\circ \psi_*}  \quad \text{\it  this extension and set } 
\quad \tilde F(f) \deff (\widetilde{F\circ\psi_*})(0,1).
\end{equation}
Remark that $\tilde F$ is indeed an extension of $F$, \ie \ that $\tilde F(f) = F(f)$
for $f\in W^{k,2}(S, H_q^n(r))$ because $\tilde F$ is a holomorphic extension of $F$ from 
$\Psi\cap W^{k,2}(S, H_q^n(r))$. We call $\tilde F$ the ``natural extension'' of $F$ .

\smallskip\noindent{\slsf Step 2.} {\it The natural extension is continuous.}  Let $f'\in W^{k,2}(S,
\Delta^q\times \Delta^n)$ be close to $f$. Construct $q$-disc $\phi_{f'*}$ and $(q+1)$-disc 
$\psi'_*$ for $f'$ as we did for $f$. Denote by $\Gamma_{\Psi}$ and 
$\Gamma_{\Psi'}$ the graphs of $\widetilde{F\circ \psi_*}$ and $\widetilde{F\circ \psi'_*}$ in 
$\Delta^{q+1}_{1+\delta}\times \calx$ correspondingly. Due to Lemma \ref{cont} all we need to prove is that 
$\Gamma_{\Psi'}$ enters to a given neighborhood of $\Gamma_{\Psi}$ provided $f'$ is sufficently close
to $f$.
% \st{ But due to Theorem this given neighborhood, say $\calv$, can be assumed to be $1$-complete.}
% \st{ Moreover, for an appropriate $\delta >0$ the graph $\Gamma_{\Psi'}$ over the Hartogs figure $H_q^1(\delta)$ as in %\ref{(hart-d)}
% enters to a given neighborhood $\calv$ of the graph $\Gamma_{\Psi}$ because $\psi'$ is close to $\psi$ when restricted to $H_q^1(\delta)$. }
% \st{And then $\Gamma_{\Psi'}\subset\calv$ by the maximum principle applied to the psh. exhausting function of $\calv$. The step is proved.}
% \color{red}
Due to Theorem \ref{royden-w} one can choose a 1-complete neighborhood $\calv$ of $\Gamma_\Psi$. Moreover, for an appropriate $\delta >0$ the graph $\Gamma_{\Psi'}$ over the Hartogs 
figure $H_q^1(\delta)$ as in (\ref{(hart-d)}) enters to the neighborhood $\calv$ of the graph 
$\Gamma_{\Psi}$ because $\psi'$ is close to $\psi$ and takes it values in $W^{k,2}(S,H^n_q(r)) $ when restricted to $H_q^1(\delta)$. A priori there is no reason that all the graph enters in $\calv$. But the maximum principle applied to the plurisubharmonic exhausting function 
of $\calv$ all the graph $\Gamma_{\Psi'}$ enters in the neighborhood $\calv$. The lemma \ref{cont} permits us to obtain the continuity of the map and the step is proved. 
\color{black}

\smallskip\noindent{\slsf Step 3.} {\it Holomorphicity.} All what is left to prove is that our 
natural extension $\tilde F$ is G\^ateaux holomorphic. Fix $f\in W^{k,2}(S,\Delta^q \times
\Delta^n)$, $g\in W^{k,2}(S,\cc^{q+n})$ and consider the complex affine line $L \deff \{f+\lambda g:  
\lambda \in \cc\}\subset W^{k,2}(S,\cc^{q+1+n})$. We need to prove that $\tilde F|_L$ is holomorphic 
in a neighborhood of zero. Fix an $\epsilon >0$ sufficiently small and consider an analytic $(q+2)$-disc 
$\theta_*:\bar\Delta^q\times\bar\Delta_{\epsilon}^2\to W^{k,2}(S,\cc^{q+n})$
%\Delta^q_{1+r}\times \Delta)$
represented by the mapping :
\begin{equation}
\eqqno(an-disk4)
\theta : (z, \lambda , \mu , s) \to \left(\frac{f_1^q(s)+\lambda g_1^q(s)-z_1}{1-(\bar f_1^q(s) +
\mu \bar g_1^q(s))z_1},...,\frac{f_q^q(s)+\lambda g_q^q(s)-z_q}{1-(\bar f_q^q(s) +
\mu\bar g_q^q(s))z_q},f^n(s)+\lambda g^n(s)\right).
\end{equation}
Denote by $\Theta = \text{Im} \theta_*$ its image.
Notice that for $\epsilon >0$ small enough the $q$-component of \eqqref(an-disk4) defines for every fixed 
$|\lambda|, |\mu | \le \epsilon$ and $s\in S$ a holomorphic imbedding of $\bar\Delta^q$ to a neighborhood of 
$\bar\Delta^q$ which is uniformly close to the authomorphism 
\[
 z\to  \left(\frac{f_1^q(s)-z_1}{1-\bar f_1^q(s)z_1},..., \frac{f_q^q(s)-z_q}
{1-\bar f_q^q(s)z_q}\right)
\]
of $\Delta^q$. Therefore for every $|\lambda |, |\mu|\le \epsilon$ the disk $\theta_*$ has values in $ W^{k,2}(S,\Delta^q_{1+r}\times \Delta^n)$ and the arguments of Step 1 can be repeated to 
$\theta (\cdot , \lambda , \mu ,\cdot)$ on the place of $\phi_f$:

Consider an analytic $(q+3)$-disk $\psi_*$ in $W^{k,2}(S,\cc^{q+n})$ 
represented by
\begin{equation}
\label{an-disk33}
\psi^t(z,\lambda, \mu, s) \deff  \psi (z,\lambda,\mu,t,s) \deff \left(\theta^q(z,\lambda,\mu, s), t \theta^n(z,\lambda,\mu, s)\right), \quad z\in \bar\Delta^q, \quad |t| \le 1+\delta , \quad s\in S
\end{equation}
for an appropriate $\delta >0$ small enough. Here $\theta^q$ and $\theta^n$ are components of $\theta = (\theta^q, \theta^n)$.
Remark that
\begin{itemize}
\item $\psi^0$ takes its values in $\Delta^q\times \{0\} \subset H_q^n(r)$ and therefore for $|t|$ small $\psi^t$ takes its values in $H_q^n(r)$ by continuity.  

\item $\psi^1 = \theta$.

\item For all $t\in \bar\Delta_{1+\delta}$ one has that $\d\Psi^t \subset W^{k,2}(S,
H_q^n(r))$.
\end{itemize}

This gives a holomorphic extension 
of $F\circ\theta_*$ from $\d\Delta^q\times \Delta^2_{\epsilon}$ to $\bar\Delta^q\times \Delta^2_{\epsilon}$.
We denote  this extension as $\hat F$. Notice now the following properties of $\theta$ and $\hat F$:

\smallskip\sli $\theta (\cdot , \lambda , \bar\lambda , \cdot) = \phi_{f+\lambda g}(\cdot , \cdot) $;

\smallskip\slii $\theta (0, \lambda , \mu , \cdot) = f(\cdot) + \lambda g(\cdot)$, in particular, this
doesn't depend on $\mu$.

\medskip\noindent
Properties \textit{iv)} and \textit{v)} imply that 
\[
\tilde F(f+\lambda g) =\left(\ext F|_{\phi_{f+\lambda g}}\right)|_{z=0} = 
\left(\ext F|_{\theta (\cdot , \lambda , \bar\lambda , \cdot)}\right)|_{z=0}
= \left(\ext F|_{\theta (\cdot , \lambda , \mu , \cdot)}\right)|_{\mu =0, z=0}=
\]
\[
=\hat F|_{\mu =0 , z=0} = \hat F(0, \lambda , 0).
\]
Here by $\ext F|_{\phi_{f+\lambda g}}$ we denote the extension of $F$
along the $q$-disc $\phi_{f+\lambda g}$ (and then taking the value of this extension 
at $z=0$). This was the definition of the natural extension. Therefore 
the first equality is justified. As for second it is justified by (\sli . 
The third equality folows from (\slii since at $z=0$ nothing depends on $\mu$ and 
$\mu =\bar\lambda$ can be replaced by $\mu =0$. But this is the extension $\hat F$
evaluated at $z=0, \lambda , \mu =0$ and the latter holomorphically depends on $\lambda$.
Therefore such is the left hand side $\tilde F (f+\lambda g)$. The holomorphicity
of $\tilde F$ is proved.

\noindent \textbf{Remark.}
The theorem just proved was stated in \cite{A-I}, see Theorem 3.1 there, for the case of compact $S$ without boundary. The step 1 of the proof was presented there as well. As for step 2 and 3 the details were missing. 
\smallskip\qed

%\bibliographystyle{plain}
%\bibliography{biblio.bib}

\ifx\undefined\bysame
\newcommand{\bysame}{\leavevmode\hbox to3em{\hrulefill}\,}
\fi

\def\entry#1#2#3#4\par{\bibitem[#1]{#1}
{\textsc{#2 }}{\sl{#3} }#4\par

\end{document}